\documentclass[oneside]{amsart}

\usepackage[letterpaper,body={12.6cm,20.5cm}, mag=1000]{geometry}
\usepackage{amssymb}
\usepackage{amsthm}
\usepackage{amscd}

\numberwithin{equation}{section}
\theoremstyle{plain}
\newtheorem{cor}[equation]{Corollary}
\newtheorem{lemma}[equation]{Lemma}

\newtheorem{prop}[equation]{Proposition}

\newtheorem*{thma}{Theorem A}
\newtheorem*{thmb}{Theorem B}
\newtheorem*{thmc}{Theorem C}
\newtheorem*{cord}{Corollary D}
\theoremstyle{definition}

\newcommand{\dlabel}[1]{\ifmmode \text{\ttfamily \upshape [#1] } \else
{\ttfamily \upshape [#1] }\fi \label{#1}}

\newcommand{\C}{\operatorname{C} }

\newcommand{\F}{\operatorname{F} }

\newcommand{\Z}{\operatorname{Z} }

\newcommand{\gen}[1]{\left < #1 \right >}

\begin{document}

\title{On subgroups generated by small classes in finite groups}

\author{Manoj K.~Yadav}

\address{School of Mathematics, Harish-Chandra Research Institute \\
Chhatnag Road, Jhunsi, Allahabad - 211 019, INDIA}

\email{myadav@hri.res.in}

\subjclass[2000]{Primary 20D25}
\keywords{Conjugacy class, Fitting subgroup, solvable, nilpotency class}

\begin{abstract}
Let $G$ be a finite group and  $M(G)$ be the subgroup of $G$ generated by 
all non-central elements of $G$ that 
lie in the conjugacy classes of the smallest size. Recently several results have been proved regarding the nilpotency class of $M(G)$ and $F(M(G))$, where $F(M(G))$ denotes the Fitting subgroup of $M(G)$. We prove some conditional results regarding the nilpotency class of $M(G)$.
\end{abstract}

\maketitle

\section{Introduction}

Let $G$ be an arbitrary finite group. An element $x$ of $G$ is said to be \emph{small} if $x$ is non-central and lies in some conjugacy class of the smallest size in $G$. Let $M(G)$ denote the subgroup of $G$ which is generated by all small elements of $G$.

For a given subgroup $H$ of $G$ and an element $x \in G$, by
 $x^H$ we denote the entire $H$-class of $x$ constisting of elements of
 the form $h^{-1}xh$, where $h$ runs over every element of $H$. For the same
 $H$ and $x$, we denote the set $\{[x, h] | h \in H\}$ by $[x, H]$, where $[x, h]$ denotes the commutator $x^{-1}h^{-1}xh$ of $x$ and $h$.
Since $x^H = x[x, H]$, it follows that $|x^H| = |[x, H]|$. By $\C_H(x)$ we denote the centralizer of $x$ in $H$. A subset $S$ of a group $G$ is said to be \emph{normal} if $g^{-1}Sg = S$ (or equivalently $g^{-1}Sg \subseteq S$) for
all $g \in G$. We write the subgroups in the lower central series of $G$ as 
$\gamma_n(G)$, where $\gamma_1(G) = G$ and $\gamma_{n+1}(G) = [\gamma_n(G),
G]$ for all $n > 1$.

In 1953, N. Ito \cite{nI53} proved that if $G$ is a finite group in which the conjugacy classes of non-central elements are all of the same size n (say), then $G$ must be nilpotent and $n = p^m$ for some prime integer $p$ and some positive integer $m$. In 2002, K. Ishikawa \cite{kI02} proved that the nilpotency class of these groups can not be more than $3$. In 2006, A. Mann \cite{aM06} proved that for any finite nilpotent group $G$, $M(G)$ is nilpotent of class at most $3$.
Recently M. Isaacs \cite[Theorem A]{mI08} generalized these results for non-nilpotent groups. He proved that if a finite group
$G$ contains a normal abelian subgroup $A$ such that $\C_{G}(A) = A$, then
$M(G)$ is nilpotent, and it has nilpotency class at most $3$. Isaacs' result shows that if $G$ is a finite group which is either nilpotent or supersolvable, then $M(G)$ is nilpotent of class at most $3$. In continuation,
A. Mann \cite{aM08} proved that $M(G)$ is nilpotent of class at most $3$ if 
either $M(G)$ is solvable and contains a normal subgroup $N$ with abelian Sylow subgroups such that $G/N$ is nilpotent, or $G$ is solvable and contains a normal subgroup $N$ with abelian Sylow subgroups such that $G/N$ is supersolvable.

The aim of this note is to prove some conditional results regarding the nilpotency class of $M(G)$ for a finite group $G$ and to give a further research direction to the topic. Let $F(G)$ denote the Fitting subgroup of a given finite group $G$. Then we prove the following theorem.

\begin{thma}
 Let $G$ be a finite group such that $\C_{G}(F(G)) \le F(G)$ and $[x, F(G)]$ is a normal subset of $F(G)$ for every small element $x$ of $G$. Then $M(G)$ is nilpotent of  class at most $2$. 
\end{thma}

To elaborate Theorem A, we would like to remark that the conditions given in this theorem are naturally satisfied in many cases. Some of such instances are mentioned below in Theorem B. The following conjecture is posed by Alexander Moreto (private communication):

\vspace{.15in}

\noindent {\bf Conjecture A.} Let $G$ be a finite solvable group with trivial center. Then every small element of $G$ lies in the center of $F(G)$.
\vspace{.15in}

In Proposition \ref{prop3}, we prove that Conjecture A is equivalent to the following conjecture:
\vspace{.15in}

\noindent {\bf Conjecture B.}  Let $G$ be a finite solvable group with trivial center. Then, for every small element $x$ of $G$, 
$[x, F(G)]$ is a normal subset of $F(G)$.
\vspace{.15in}

Following \cite{cS62}, we say that a subgroup $N$ of a group $G$ is \emph{$c$-closed} if any two elements of $N$, which are conjugate in $G$, are also conjugate in $N$. More results on this topic can be found in \cite{cS68}. Following \cite{TM04}, we say that a group $G$ is \emph{flat} if $[x, G]$ is a subgroup of $G$ for all $x \in G$. The following theorem elaborates the usefulness of Conjecture B.

\begin{thmb}
Let $G$ be a finite solvable group with trivial center. Then Conjecture A holds true if one of the following holds:
\begin{enumerate}
\item $F(G)$ is ablein;
\item $\gamma_2(G) \cap F(G) \le \Z(F(G))$;
\item $F(G)$ is $c$-closed in $G$ and the nilpotency class of $F(G)$ is $2$;
\item $M(G) \le \F(G)$ and $F(G)$ is flat. 
\end{enumerate}
\end{thmb}

Notice that, in all given cases of Theorem B, $[x, F(G)]$ is a normal subset of $F(G)$ for every small element $x$ of $G$. Thus Conjecture B holds true. Now the proof of Theorem B follows from the equivalence of Conjectures A and B. We would like to remark that Conjcture A can also be stated in the following equivalent form, which can be proved easily by Proposition \ref{prop2} given below:
\vspace{.15in}

\noindent {\bf Conjecture C.} Let $G$ be a finite solvable group with trivial center. Then, for every small element $x$ of $G$, 
$[x, F(G)] \subseteq \Z(F(G))$.
\vspace{.15in}

 In the following theorem, we generalize (at least in principle) Isaacs' result \cite[Theorem A]{mI08}.

\begin{thmc}
Let $G$ be a finite group that contains a normal subgroup $A$ such that 
$\C_{G}(A) \le A$ and $[A, x]$ is a normal subset of $A$ for all small elements $x \in G$. 
Then $M(G)$ is nilpotent, and it has nilpotency class at most $3$. 
\end{thmc}

Notice that $[A, x]$ is a normal subset of $A$ if it is contained in $\Z(A)$. Thus Theorem A gives the following result, which in turn readily gives Isaacs' result.

\begin{cord}
Let $G$ be a finite group that contains a normal subgroup $A$ such that 
$\C_{G}(A) \le A$ and $[A, x]$ is contained in the center of $A$ for every small element $x$ of $G$. 
Then $M(G)$ is nilpotent of class at most $3$. 
\end{cord}

\section{Proofs}

We start with the following lemma which is a generalization of Lemma 1 of Martin Isaacs \cite{mI08}.  
 
\begin{lemma}\label{lemma2}
Let $G$ be a finite group and $K$ be a normal subgroup of $G$. Let $x \in G -
\Z(G)$ such that $[x, K]$ is a normal subset of $K$. Then, for any $y \in [x, K]$, $|\C_{G}(y)| > |\C_{G}(x)|$.
\end{lemma}
\begin{proof}
If $y = 1$, then the result follows trivially, since $x$ is non-central
element of $G$. So assume that $y \neq 1$. Now let $H = K \C_{G}(x)$. Notice that
$y \in H$,  $|x^K| = |x^H|$ and $\C_{H}(x) = \C_{G}(x)$. Therefore it suffices to show that
$|H : \C_{H}(y)| < |H : \C_{H}(x)|$.

We claim that $[x, K]$ is a normal subset of $H$. For, let $h \in H$ and 
$[x, k] \in [x, K]$. Then $h = uv$, for some $u \in K$ and $v \in \C_G(x)$.
So $w := h^{-1} [x, k]h = v^{-1}u^{-1}[x, k]uv = v^{-1} [x, k_1]v$ for some
$k_1 \in K$, since $[x, K]$ is a normal subset of $K$. Now 
$w = [v^{-1}xv, v^{-1}kv] = [x, v^{-1}kv] = [x, k_2]$ for $k_2 = v^{-1}kv \in
K$, since $v \in \C_G(x)$ and $K$ is normal in $G$. Thus $w \in [x, K]$. 
This proves our claim, i.e., $[x, K]$ is a normal subset of $H$. Notice
that $1 = [x, 1] \in [x, K]$. 
Since $y$ is a non-trivial element of $[x, K]$, which is normal in $H$, 
it follows that $y^H$ consists of non-identity elements of $[x, K]$. Thus  $|y^H| < |[x, K]|$. 
Since $|[x, K]| = |x^K|$ and $|x^K| = |x^H|$, we get 
\[|H : \C_{H}(y)| = |y^H| < |[x, K]| = |x^K| = |x^H| =  |H : \C_{H}(x)|.\]
 This completes the proof of the lemma. \hfill $\Box$ 

\end{proof}

As an application  of Lemma \ref{lemma2} we have the following two propositions.

\begin{prop}\label{prop1}
Let $K$ be a normal subgroup of an arbitrary finite group $G$ such that  
$[x, K]$ is a normal subset of $K$ for all small elements $x$ of $G$. Then  $[M(G), K] \le \Z(G)$.
\end{prop}
\begin{proof}
It suffices to prove that $[x, K] \subseteq \Z(G)$ for all small elements $x$ of $G$. Let $x$ be an arbitrary  small element of $G$ such that $[x, K]$ is normal in $K$ as a subset. Let $y \in [x, K]$ be an arbitrary element. Then it follows from Lemma \ref{lemma2} that 
$|y^G| < |x^G|$. Hence $|y^G| = 1$, and hence $y \in \Z(G)$. 
\hfill $\Box$

\end{proof}

\begin{prop}\label{prop2}
 Let $K$ be a normal subgroup of an arbitrary finite group $G$ and $x$ be a small element of $G$. Then $[x, K] \subseteq \Z(K)$ if and only if $[x, K] \subseteq \Z(G)$.
\end{prop}
\begin{proof}
 Suppose that $[x, K] \subseteq \Z(G)$. Since $K$ is normal in $G$, $[x, K] \subseteq K$. Thus $[x, K] \subseteq K \cap \Z(G) \le \Z(K)$.
Conversely, suppose that $[x, K] \subseteq \Z(K)$. Then $[x, K]$  is a normal subset of $K$. Let $y \in [x, K]$ be an arbitrary element. Then it follows from Lemma \ref{lemma2} that $|y^G| < |x^G|$. Hence $|y^G| = 1$, and hence $y \in \Z(G)$. 
\hfill $\Box$

\end{proof}

\vspace{.15in}

Now we are ready to prove our results stated in the introduction.

\vspace{.13in}

\noindent \emph{Proof of Theorem A.}
By the given hypothesis, we have $\C_{G}(F(G)) \le F(G)$. Let $\bar{G} = G/\Z(G)$. We claim that $\C_{\bar{G}}(F(\bar{G})) \le F(\bar{G})$. Let $\pi$ be the natural projection from $G$ to $\bar{G}$. Let $C$ be the inverse image of $\C_{\bar{G}}(F(\bar{G})) = \C_{\bar{G}}(F(G)/\Z(G))$ under $\pi$. Then it follows that $C$ is a normal subgroup of $G$ containing $\C_{G}(F(G))$ and satifying $[C, F(G)] \le \Z(G)$. Now $\hat C= C/\C_G(F(G))$ acts
faithfully as automorphisms of $F(G)$ centralizing both $\Z(G)$ and
$F(G)/\Z(G)$. It follows that $\hat C$ is abelian, and acts on the
abelian group $\C_G(F(G)) = \Z(F(G))$. Furthermore, it centralizes both
$\Z(F(G))/\Z(G)$ and $\Z(G)$. This implies that $C$ is nilpotent as well
as normal in $G$.  Therefore $C \le F(G)$ and $C/\Z(G) = \C_{\bar
G}(F(\bar G)) \le F(G)/\Z(G) = F(\bar G)$. This proves our claim.

Let $x$ be an arbitrary small element of $G$. Then by the given hypothesis we know that $[x, F(G)]$ is a normal subset of $F(G)$. Thus it follows from the proof of Proposition \ref{prop1} that  $[x, F(G)] \subseteq \Z(G)$. So the image $\bar{x} = x\Z(G)$ of $x$ under $\pi$ satisfies
 \[  [\bar x, F(\bar G)]= [x\Z(G), F(G)/\Z(G)] = [x,F(G)]\Z(G)/\Z(G) = \bar 1.\]
Hence $\bar x \in \C_{\bar G}(F(\bar G)) = \Z(F(\bar G))$, since $\C_{\bar{G}}(F(\bar{G})) \le F(\bar{G})$. Therefore $x \Z(G) \in F(\bar G) =
F(G)/\Z(G)$ and $x \in F(G)$. Since $[x, F(G)] \subseteq \Z(G) \le \Z(F(G))$, it now follows that $x$ belongs to the second center of $F(G)$.  Now it is clear that the nilpotency class of $M(G)$ is at most $2$. This completes the proof of the theorem. \hfill $\Box$

\vspace{.13in}

We would like to remark here that the first hypothesis in Theorem A, i.e., $\C_{G}(F(G)) \le F(G)$, is naturally satisfied in all finite solvable groups $G$. So, as a corollary we have

\begin{cor}
 Let $G$ be a finite solvable group such that $[x, F(G)]$ is a normal subset of $F(G)$ for every small element $x$ of $G$. Then $M(G)$ is nilpotent of  class at most $2$.
\end{cor}

The following proposition proves the equivalence of the Conjectures $A$ and  $B$ stated in the introduction.

\begin{prop}\label{prop3}
 Let $G$ be a finite solvable  group with trivial center and $x$ be a small element of $G$. Then the following are equivalent:
\begin{enumerate}
\item  $x \le \Z(F(G))$;
\item $[x, F(G)]$ is a normal subset of $F(G)$.
\end{enumerate}
\end{prop}
\begin{proof}
(1) trivially implies (2). So suppose that (2) holds. Then it follows from the proof of Proposition \ref{prop1} that   $[x, F(G)] \le \Z(G) = 1$. This shows that $x$ centralizes $F(G)$. Since $\C_{G}(F(G)) \le F(G)$, it follows that $x \le \Z(F(G))$ and (1) holds. This completes the proof of the proposition.  \hfill $\Box$

\end{proof}

The following proof of Theorem C is similar to the proof of \cite[Theorem A]{mI08}.
\vspace{.13in}

\noindent \emph{Proof of Theorem C.} Write $M = M(G)$. Since $[A, x]$ is a normal subset of
$A$, by Proposition \ref{prop1} we get $[A, M] \le \Z(G)$. 
Thus $[A, M, M] = 1$. Now by the three-subgroup lemma, it follows that
$\gamma_2(M) \le \C_{G}(A) \le A$. Thus $\gamma_4(M) = [\gamma_2(M), M, M] \le
[A, M , M] = 1$. This proves that $M$ is  nilpotent of class at most $3$,
which completes the proof of the theorem. \hfill $\Box$

\vspace{.13in}

A group $G$ is said to be of \emph{conjugate rank} $1$ if all non-central elements of $G$ have the same conjugacy class size. 

Finally we prove the following proposition which is not related to the previous discussion directly, but is a nice application of Proposition \ref{prop1}.

\begin{prop}
 Let $G$ be  a finite $p$-group of conjugate rank $1$. Then $G$ is flat if and only if the nilpotency class of $G$ is $2$. 
\end{prop}
\begin{proof}
Since every group of nilpotency class $2$ is flat, we only need to prove the only if part of the proposition. So let $G$ be flat and $x \in G - \Z(G)$. Thus $[x, G]$ is a subgroup and hence a normal subgroup of $G$. Now it follows from the proof of Proposition \ref{prop1} that  $[x, G] \le \Z(G)$. Since $\gamma_2(G) = \gen{[x, G] | x \in G - \Z(G)}$, it follows that $\gamma_2(G) \le \Z(G)$. This completes the proof of the proposition.  \hfill $\Box$ 

\end{proof}

\noindent{\bf Acknowledgements.} I thank Prof. E. C. Dade for his useful comments and suggestions.

\end{document}